\documentclass[12pt, reqno]{amsart}
\usepackage{geometry}                % See geometry.pdf to learn the layout options. There are lots.
\usepackage{graphicx}
\usepackage{booktabs}
\usepackage{amssymb}
\usepackage{epstopdf}
\usepackage{hyperref}

\makeatletter
\def\@setthanks{\vspace{-\baselineskip}\def\thanks##1{\@par##1\@addpunct.}\thankses}
\makeatother

\newtheorem{theorem}{Theorem}
\newtheorem{lemma}[theorem]{Lemma}
\newtheorem{definition}[theorem]{Definition}
\newtheorem{corollary}[theorem]{Corollary}
\newtheorem{proposition}[theorem]{Proposition}

\title[From MFG To NSE]
{From Mean Field Games To Navier-Stokes Equations}

\author{ Tao Luo}
\author{Qingshuo Song}
\thanks{
\noindent 
T. Luo is with the Department of Mathematics, City University of Hong Kong, \url{taoluo@cityu.edu.hk}\\
Q. Song is with the Department of Mathematical Sciences, Worcester Polytechnic Institute, 
\url{qsong@wpi.edu}.
}

\begin{document}
\maketitle
%\section{}
%\subsection{}
\begin{abstract}
This work establishes the equivalence between Mean Field Game and a class of  PDE systems closely related to compressible Navier-Stokes equations.
The solvability of  
the PDE system via the existence of the Nash Equilibrium of the Mean Field Game
is provided under a set of conditions.
\end{abstract}

\section{Introduction}

Global-in-time existence  of strong solutions to the incompressible Navier-Stokes equations in 3-dimensions is a well known open question  to decades as one of seven Clay Millennium prizes (\cite{fefferman}):
\begin{equation}
\label{eq:nse1}
\begin{array}
{ll}
\partial_t v =   \nu \Delta v -( v \cdot \nabla) v - \nabla p + f, \quad t>0, \ x\in \mathbb M^3, \\
\nabla \cdot v = 0, \\
v(0, x) = v_0(x).
\end{array}
\end{equation}
In the above, the constant $\nu>0$ is the viscosity coefficient,   the unknowns are the velocity $v$ and the pressure $p$, 
and the state domain $\mathbb M^3$ is either $\mathbb R^3$,  $\mathbb T^3$ or a bounded domain $\mathbb{D}$ in $\mathbb{R}^3$ with a boundary (in this case, the non-slip boundary condition $v=0$ on $\partial \mathbb {D}$ is prescribed). 
The equations describe
the motion of an incompressible viscous fluid filling $\mathbb M^3$. The vector $v$ is the velocity of
the fluid element and the scalar quantity $p$ measures the pressure exerted on the fluid element. For the above problem in 3-spatial dimensions, the local-in-time existence of strong solutions and global-in-time existence of weak solutions is classical. However, the global-in-time existence of strong solutions (or the global-in-time regularity of solutions) remains an outstanding open problem in 3-spatial dimensions (\cite{fefferman}). This is in sharp contrast to the 2-spatial dimensional case. The above classical results are due to J. Leray (\cite{leray1, leray2}), E. Hopf (\cite{hopf}), and O. A. Ladyzhenskaya (\cite{ladyz1,ladyz2}), J. L. Lions-G. Prodi (\cite{LP}), J. Serrin (\cite{serrin}). See also the details and some recent developments in monographs of Constantin-Foias (\cite{constantine}), 
Foias-Manley-Rosa-Temam (\cite{Foias}), P. L. Lions (\cite{Lions 1}), Gilles-Rieusset (\cite{GR18}) and Temam(\cite{temam}). In 3-spatial dimensions, 
the celebrated partial regularity result of Caffarelli-Kohn-Nirenberg  (\cite{CKN}) states that the one-dimensional Hausdorff measure 
of the singular set of the suitable weak solution is zero (see also F. Lin (\cite{lin}) for a new proof), which strengthened the previous result of Scheffer's result (\cite{scheffer1, scheffer2, scheffer3, scheffer4}), see also Foias-Temam (\cite{FT}) for the interesting work in this direction. 

The above mentioned results are for incompressible  Navier-Stokes equations. In a more general setting, one has the compressible Navier-stokes equations modelling the motion of compressible viscous fluids, which reads, for the isentropic fluids (\cite{Lions 2, GR18}
\begin{equation}
\label{eq:nse2}
\begin{array}
{ll}
(\partial_t + v\cdot \nabla) v = - \rho^{-1} \nabla p + \rho^{-1}  \nu \Delta v  + \rho^{-1}  f \\
\partial_t \rho + \nabla (\rho v) = 0,\\
v(0, x) = v_0(x), \ \rho(0, x) = \rho_0.
\end{array}
\end{equation}

In this case, the pressure $p$ is a given function of the density $\rho$,  while the unknowns are the density $\rho$ and the velocity field $v$
with given initial conditions at time $0$. Two equations are termed the momentum equation and the mass conservation, respectively. Here, we have simplified the viscosity term for the simplicity of presentations. 

Indeed, the hydrodynamic equation \eqref{eq:nse2} is equivalent to NSE \eqref{eq:nse1} if the density $\rho$ is a nonzero 
constant (the underlying fluid dynamic is incompressible in this case). In the literature, the  hydrodynamic equation \eqref{eq:nse2}
is termed compressible NSE. Formally, when the motion is slow, i.e., when the Mach number (the ratio of fluid velocity and sound speed) is very small, the incompressible model \eqref{eq:nse1} can be used to approximate the compressible model \eqref{eq:nse2}. 

Compared with the incompressible NSE \eqref{eq:nse1}, much less is known about the compressible NSE \eqref{eq:nse2}.
There has been a very satisfied well-posedness theory for the compressible Navier-stokes equations in 1-spatial dimensions, for instance, 
one may refer to \cite{ka, ks}, by using the order structure of the real line and some special transport properties along the particle path. However, the situation is quite different and difficult in higher spatial dimensions of 2 or 3.  For the general initial data which do not have to be small perturbations of constant states, the global-in-time existence of {\it weak} solutions for the compressible Navier-Stokes equations in 3-spatial dimensions 
was proved by P. L. Lions (\cite{Lions 2}), for suitable conditions, which was a breakthrough in the study of fluid PDEs. 
Meanwhile, many problems of fundamental importance remain open, for instance, the regularity and uniqueness of the weak solutions constructed in \cite{Lions 2}.

In this paper, we propose a variation of the compressible NSE \eqref{eq:nse2}. 
\begin{itemize}
\item
By Newton's second law, the right hand side of momentum equation of the compressible NSE \eqref{eq:nse2} is given as the ratio between the total force $ -  \nabla p +   \nu \Delta v  +   f$ and the density $\rho$. In our variation,
the ratio on the right-hand side is generalized to a given functional of the total force term depending on the density $\rho$.
\item 
The mass conservation of \eqref{eq:nse2} describes the density flow $\rho$ at the velocity $v$.
We replace this equation by a Fokker-Planck equation (FPK) with the diffusive coefficient 
tied up with the viscosity term $\nu[\rho]$, which describes the white noise perturbed particle density.
Moreover, we impose the terminal condition on the density.
\end{itemize}
The resulting system becomes
\begin{equation}
\label{eq:nse3}
\begin{array}
{ll}
(\partial_t + v\cdot \nabla) v = -  \nabla p [\rho] +  \nu [\rho] \Delta v  +  f [\rho],  \ (t, x)\in [0, T] \times \mathbb M^3\\
\partial_t \rho + \nabla \cdot (\rho v) + \nu [\rho]   \Delta \rho = 0, \\
v(0, x) = v_0(x), \ \rho(T, x) = \rho_T(x).
\end{array}
\end{equation}
%\begin{itemize}
%\item
%Note that $\rho$ is a probability density with the total mass 1 and there is no uniform probability density in the whole space $\mathbb R^3$.  Therefore, the state space $\mathbb M^3$ will be taken for Torus $\mathbb T^3$ in the rest of the paper.
%\end{itemize}

In this paper, we consider a class of PDE systems fitting into the framework formulated by \eqref{eq:nse3}.
Our main contribution to this work can be summarized as follows:
Although the proposed PDE system \eqref{eq:nse3} does not strictly follow the original physical interpretations, we provide an alternative interpretation from Mean Field Game (MFG) theory with some conditions. 
To the best of our knowledge, the connection between the Navier-Stokes equation and the Mean Field Game has not been established in the literature.  Offering possible insights from the point of view of the Mean Field Game to the study of compressible and incompressible Navier-Stokes equations is one of the motivations of the present paper. 
For more details on MFG, we refer \cite[Lasry and Lions]{LL07a} and \cite[Huang, Caines, and Malhame]{HCM06},  
\cite[Carmona and Delarue]{CD18I}.

Meanwhile, as far as a torus domain is concerned, most MFG analysis remains on an analytical approach in the literature. 
%This below,  
Our proof of the regularity of the proposed NSE very much relies on the probabilistic analysis of HJB equations on a torus.
Indeed, our approach opens up the possibility to use another probabilistic approach, 
such as the approach based on Forward-Backward Stochastic Differential Equation (FBSDE)
to analyze MFG and further different kinds of NSE on a torus. We refer \cite[Ma and Yong]{MY99b} for excellent exposition on FBSDEs and their relation to stochastic control problems.

Our proof below relies on Schauder's fixed point theorem, which is closely related to \cite{CHHJS20}, while the sufficient condition in this paper on the existence is weaker by relaxing Lipschitz continuity of $(p, h)$ as of uniformly continuous function in the form of \eqref{eq:a2}, see also in \cite{Car13} for its comparison.

\section{The main result}
In this section, we describe the precise problem setting and the main result. 
An example is also provided for the illustration purpose.
\subsection{Notations} \label{s:notations}
To proceed, we will introduce the following notions. 
We denote
the  $3$-torus state space by $\mathbb T^3 = \mathbb R^3/\mathbb Z^3$.
For $x\in \mathbb R^3$, let $\pi(x)$ be the coset of $\mathbb Z^3$ that contains $x$, i.e. 
$$\pi(x) = x + \mathbb Z^3.$$
A canonical metric on $\mathbb T^3$ can be induced from the Euclidiean metric by
$$|\pi(x) - \pi(y)|_{\mathbb T^3} = \inf\{|x - y - z|: z\in \mathbb Z^3\}.$$
For a function $f: [0, T]\times \mathbb T^3 \mapsto S$ with its range in a Banach space $S$, we define,  for $\delta_1, \delta_2\in [0,1]$ 
$$|f|_0 = \sup_{(t, x) \in  [0, T]\times \mathbb T^3} |f(t,x)|, \
[f]_{\delta_1, \delta_2} = \sup_{0\le t\neq t' \le T, x\neq x' \in\mathbb T^3}
\frac{|f(t,x) - f(t', x')|}{|t-t'|^{\delta_1} + |x- x'|^{\delta_2}}.
$$
We denote $f\in C^{m+\delta_1, n+\delta_2}( [0, T]\times \mathbb T^3, S)$ for nonnegative integers $m$ and $n$, if it satisfies
$$|f|_{m+\delta_1, n+\delta_2} = |f|_{m, n} + [\partial_t^m f]_{\delta_1, \delta_2} +
\sum_{|\alpha| = n} [D_x^\alpha f]_{\delta_1, \delta_2} <\infty,$$
where
$$
|f|_{m, n} = \sum_{i \le m} |\partial_t^i f|_0  +
\sum_{|\beta|\le n} 
|D_x^\beta f|_0.$$
In this paper, we use $S$ for Euclidean spaces $\mathbb R$ or $ \mathbb R^3$. Moreover, the   $C^{m+\delta_1, n+\delta_2}( [0, T]\times \mathbb T^3)$ will be used for a short notation of
$C^{m+\delta_1, n+\delta_2}( [0, T]\times \mathbb T^3, \mathbb R)$.
We also denote the closed ball of radius $r$ with center zero in the space  $C^{m+ \delta_1, n+\delta_2}$ by 
$B^{m+ \delta_1, n+\delta_2}(r)$, i.e.
$$B^{m+ \delta_1, n+\delta_2}(r) = \{f \in C^{m+ \delta_1, n+\delta_2}: |f|_{m+ \delta_1, n+\delta_2} \le r\}.$$
Note that,  for any $m'\le m$ and $n'\le n$, the set $B^{m+ \delta_1, n+\delta_2}(r)$ is a compact set in the space 
$C^{m', n'}([0, T]\times \mathbb T^3, S)$ by the embeding theory. 

Let's denote the space of all probability distributions on $\mathbb T^3$ by $\mathcal P = \mathcal P(\mathbb T^3)$.
For a random variable $X: \Omega \mapsto \mathbb T^3$
on a probability space $(\Omega, \mathcal F, \mathbb P)$, 
we denote by $\mathcal L(X) = \mathbb P X^{-1}$ for its 
pushforward measure on $\mathbb T^3$. 
In particular, any random variable $X$ has a finite $p$th moment $\mathbb E[ |X|^p ] \le 1.$
Hence, any probability distribution has a finite $m$-th moment  for any $m\ge 1$, and for the simplicity,
we use  the 1-Wasserstein metric $d_1$ to the space $\mathcal P$:
$$d_1(\mu, \nu) = \inf_{\pi \in \Pi(\mu, \nu)} 
\left( \int_{\mathbb T^3 \times \mathbb T^3} |x - y| d\pi(x, y) \right),$$
where $\Pi(\mu, \nu)$ is the collection of joint distribution on the product space $\mathbb T^3 \times \mathbb T^3$ with $\mu$ and $\nu$ as marginal distributions.

We say a probability measure flow $\rho\in C^0([0, T], \mathcal P)$ if $\rho$ is a continuous mapping of 
$[0, T] \mapsto \mathcal P$ with respect to the metric $d_{1}(\cdot, \cdot)$. Then, $C^0([0, T], \mathcal P)$ is a complete 
metric space defined by
\begin{equation}
\label{eq:rho01}
d_{1T} (\rho_1, \rho_2) = \sup_{0\le t\le T} d_1(\rho_1(t), \rho_2(t)).
\end{equation}
We set  
$$|\rho|_{1T} = d_{1T}  (\rho, \bar \delta_0),$$ 
where 
$\bar \delta_0 (t) \equiv \delta_0$ for all $t$ with dirac distribution $\delta_0$.

We also define a subset of $C^0([0, T], \mathcal P)$: 
We say $\rho \in C^\delta([0, T], \mathcal P)$ for some $\delta \in (0,1)$  if $\rho$ satisfies additional condition
$$|\rho|_{1T, \delta} := |\rho|_{1T} + [\rho]_{1T, \delta}<\infty,$$
where
$$
[\rho]_{1T, \delta} := \sup_{0\le t<t'\le T} \frac{d_1(\rho(t), \rho(t'))}{|t-t'|^\delta}.
$$
Finally, we also denote by $\Lambda$ the collection of all functions $\lambda: [0, \infty) \mapsto \mathbb R^+$ satisfying
$$\lambda'(x) >0,  \forall x>0; \ \lim_{x\to \infty} \lambda(x) = \infty.$$
%We also denote a subset of $\Lambda$ given by
By $\Lambda_0$, we denote a subset of $\Lambda$ given by
$$\Lambda_0 = \{\lambda \in \Lambda: \lambda(0) = 0.\}$$
%with $\lambda(0) = 0$.

\subsection{The problem setting and the main result}

To focus on the main issue on the connection between NSE and MFG, 
we impose the following conditions to the system \eqref{eq:nse3} throughout this paper:
\begin{itemize}
\item[(H)] The viscosity coefficient is $\nu[\rho] = 1/2$ and the scaled external force functional is $f[\rho] = 0$. Furthermore, 
$v_0[\rho] =  \nabla h[\rho]$ is given as a conservative vector field for some smooth function $h$, and 
the $\rho_T$ is given as a smooth probability density $\mu$. 
%\item
Finally, the state space $\mathbb M^3$ will be taken for Torus $\mathbb T^3$ mainly for the simplicity of the proof at the infinity. 

%We also use time reversal mapping $\Gamma(t) = T -t$ to the system.
\end{itemize}

Therefore, our goal in this paper is to consider the solvability of $(\rho, v)$ from the following equation:
\begin{equation}
\label{eq:nse4}
\begin{array}
{ll}
(\partial_t + v\cdot \nabla) v = -  \nabla p [\rho] +  \frac 1 2  \Delta v  ,  \ (t, x)\in [0, T] \times \mathbb T^3\\
\partial_t \rho + \nabla \cdot (\rho v) + \frac 1 2    \Delta \rho =0, \\
v(0, x) =  \nabla h[\rho](x), \ \rho(T, x) = \mu(x).
\end{array}
\end{equation}

Our main result is the solvability given below.
%In this proof, $\kappa(\cdot)$ will be used as a generic positive increasing function.
%Next, we will present our main result.
\begin{theorem}
\label{t:main}
Suppose $p$ and $h$ satisfies, 
\begin{equation}
\label{eq:a1}
|p[\rho]|_{0, 2} + |h[\rho]|_4 \le \kappa, \ \forall \rho\in C^{1/2} ([0, T], \mathcal P), 
\end{equation}
\begin{equation}
\label{eq:a3}
|p[\rho]|_{1/2, 2} < \infty, \forall \rho\in  C^{1/2} ([0, T], \mathcal P), 
\end{equation}
and
\begin{equation}
\label{eq:a2}
|p[\rho_1] - p[\rho_2]|_{0, 1} + |h[\rho_1] - h[\rho_2]|_4 \le \lambda_0 (d_{1T} (\rho_1, \rho_2)),  
\ \forall \rho_1, \rho_2 \in C^{1/2} ([0, T], \mathcal P)
\end{equation}
for some $\kappa>0$ and $\lambda_0 \in \Lambda_0$.
Then, 
there exists a classical 
solution %$(\rho, v)$ 
for
the system \eqref{eq:nse4}.
\end{theorem}

In the above theorem, \eqref{eq:a1}-\eqref{eq:a3}-\eqref{eq:a2} imposes sufficient conditions on the mappings 
$$p: C^{1/2} ([0, T], \mathcal P) \mapsto C^{1/2, 2}([0, T]\times \mathbb T^3, \mathbb R), \ h:  C^{1/2} ([0, T], \mathcal P) \mapsto C^{4}(\mathbb T^3, \mathbb R).$$
In particular, the condition \eqref{eq:a2} can be interpreted as the uniform continuity of $\rho$ and $h$ with modulus of continuity $\lambda_0$, when we apply lower topologies to their respective domains and ranges. 
The proof of Theorem \ref{t:main} will be relegated to the next section. 
\subsubsection{Example}

For the illustration purpose, we will take the following forms of the functions $p$ and $h$ as examples:
\begin{equation}
\label{eq:p2}
p[\rho](t, x)  = \int_{\mathbb T^3} \bar p( x, y) \rho(t, y) dy + \hat p(x),
\end{equation}
\begin{equation}
\label{eq:h2} 
h[\rho](x) = \int_{\mathbb T^3} \bar h( x, y) \rho_T(y) dy + \hat h(x),
\end{equation}
for 
some smooth enough functions 
$\bar p, \bar h: \mathbb T^3 \times \mathbb T^3 \mapsto \mathbb R$ and $\hat p, \hat h: \mathbb T^3 \mapsto \mathbb R$ on their respective domains.
In this below, we will prove \eqref{eq:p2} and \eqref{eq:h2} satisfies  \eqref{eq:a1}-\eqref{eq:a3}-\eqref{eq:a2}. 

From the definition of $p$, 
$p[\rho]$ is differentiable to any order in  the variable $x$ due to the smoothness of $\bar p$ and $\hat p$, 
and in particular it belongs to $C^{0, 2}$. Similarly, $h[\rho]$ belongs to  $C^4$ from its definition.
The $1/2$-H\"older regularity in $t$ can be seen from the following inequality:
\begin{equation}
\label{eq:p_rho}
\begin{array}
{ll}
|p[\rho](t_1, x) - p[\rho](t_2, x)| & = 
|\int \bar p(x, y) (\rho(t_1, y) - \rho(t_2, y)) dy| \\
&\le |D_y \bar p|_0 d_1(\rho(t_1), \rho(t_2)) \\
& \le |D_y \bar p|_0  [\rho]_{1T, 1/2} |t_1 - t_2|^{1/2}. 
\end{array}
\end{equation}
In the above, $D_y \bar p$ is a gradient vector for $y \mapsto \bar p(x, y)$.
\begin{itemize}
\item
Therefore, $p[\rho] \in C^{1/2, 2}$ holds and this implies \eqref{eq:a3}.
It's important to note that, the above estimate indicates $t\mapsto p[\rho](t)$ is $1/2$-H\"older, 
but its norm can be arbitrarily large since $ |D_y \bar p|_0  [\rho]_{1T, 1/2}$ is increasing to 
infinity as $[\rho]_{1T, 1/2}$ increases.

\end{itemize}
Next, for any multiindex $\alpha$, we have
$$
|D_x^\alpha p[\rho]|_0 = 
\Big| \int D^\alpha_x \bar p(x, y) \rho(t, y) dy + D_x^\alpha \hat p(x)|_0 
\le | D^\alpha_x \bar p|_0 + |D_x^\alpha \hat p|_0.
$$
So the estimate of 
$$|p[\rho]|_{0,2} \le |\bar p|_{2, 0} + |\hat p|_2$$ 
is followed by applying the above inequality to $|\alpha|\le 2$.

Similarly, one can write 
$$
|D_x^\alpha h[\rho]|_0 = 
\Big| \int D^\alpha_x \bar h(x, y) \rho(T, y) dy + D_x^\alpha \hat p(x) \Big|_0 
\le | D^\alpha_x \bar h|_0 + |D_x^\alpha \hat h|_0, 
$$
which implies 
$$|h[\rho]|_{4} \le |\bar h|_{4, 0} + |\hat h|_4.$$ 
Hence, 
\begin{itemize}
\item
the uniform boundedness \eqref{eq:a1} holds with a choice of
$$\kappa = |\bar p|_{2, 0} + |\hat p|_2 + |\bar h|_{4, 0} + |\hat h|_4.$$
\end{itemize}
In this below, we collect  inequalities  for the proof of  the uniform continuity \eqref{eq:a2}: 
\begin{enumerate}
\item 
$$
\begin{array}
{ll}
|p[\rho_1] - p[\rho_2]|_0 & = 
\sup_{0\le t\le T, x\in \mathbb R} 
|p[\rho_1](t, x) - p[\rho_2](t, x)| \\
& =  \sup_{0\le t\le T, x\in \mathbb R} 
|\int \bar p(x, y) (\rho_1(t, y) - \rho_2(t, y)) dy| \\
& \le |D_y \bar p|_0 d_{1T}(\rho_1, \rho_2).
\end{array}
$$
\item 
$$
\begin{array}
{ll}
|\partial_{x_j} p[\rho_1] - \partial_{x_j} p[\rho_2]|_0 & = 
\sup_{0\le t\le T, x\in \mathbb R} 
|\partial_{x_j} p[\rho_1](t, x) - \partial_{x_j} p[\rho_2](t, x)| \\
& =  \sup_{0\le t\le T, x\in \mathbb R} 
|\int \partial_{x_j} \bar p(x, y) (\rho_1(t, y) - \rho_2(t, y)) dy| \\
& \le |D_y \partial_{x_j} \bar p|_0 d_{1T}(\rho_1, \rho_2), \ \forall j = 1, 2, 3.
\end{array}
$$
\item Similarly, we have
$$|h[\rho_1] - h[\rho_2]|_0 \le |D_y \bar h|_0 d_{1T}(\rho_1, \rho_2)$$
and 
$$
|D^\alpha h[\rho_1] - D^\alpha h[\rho_2]|_0 \le  |D_y D^\alpha_x\bar h|_0 d_{1T}(\rho_1, \rho_2), \ \forall \alpha.
$$
\end{enumerate}
Therefore, \begin{itemize}
\item
the continuity \eqref{eq:a2} holds with 
$$\lambda_0 (x) = (|\bar p|_{1,1} + |\bar h|_{4,1}) x.$$
\end{itemize}
In view of Theorem \ref{t:main}, we conclude that 
\begin{corollary}
Suppose the functions $\bar p, \bar h, \hat p, \hat h$ satisfies 
$$\bar p \in C^{2,1}(\mathbb T^3 \times \mathbb T^3), \bar h\in C^{4,1}(\mathbb T^3 \times \mathbb T^3), \hat p \in C^2(\mathbb T^3), 
\hat h \in C^4(\mathbb T^3).$$
Then, there exists a classical solution forthe system \eqref{eq:nse4} with $(p, h)$ given in the form of  \eqref{eq:p2} and \eqref{eq:h2}.
\end{corollary}

\section{Analysis}

This section is devoted to the proof of the main result given by Theorem \ref{t:main} on the solvability of \eqref{eq:nse4}.
First, if we apply the time reversal mapping $t \mapsto T -t$ to the system \eqref{eq:nse4}, it 
is equivalent to the  system \eqref{eq:fpk2}-\eqref{eq:momentum1} given by
%Therefore, our problem is to consider the solvability of the system
%of unknown $(\rho, v)$ given by
\begin{equation}
\label{eq:fpk2}
\begin{array}
{ll}
\partial_t \rho - \nabla \cdot (\rho v) = \frac 1 2 \Delta \rho, &\hbox{ on } (0, T)\times \mathbb T^3 \\
\rho(0, x) = \mu(x), & \hbox{ on } \mathbb T^3,
\end{array}
\end{equation}
and 
\begin{equation}\label{eq:momentum1}
\begin{array}
{ll}
\partial_t v - v\cdot \nabla v + \frac 1 2 \Delta v = \nabla p[\rho] & \hbox{ on } (0, T)\times \mathbb T^3 \\
v(T, x) =  \nabla h[\rho](x) & \hbox{ on } \mathbb T^3,
\end{array}
\end{equation}
That is,
\begin{itemize}
\item
A pair $(\rho, v)(t, x)$ solves \eqref{eq:fpk2}- \eqref{eq:momentum1} if and only if $(\rho, v)(T-t, x)$ solves the PDE system \eqref{eq:nse4}.
\end{itemize}
Meanwhile, we consider HJB equation given by
\begin{equation}\label{eq:hjb1}
\begin{array}
{ll}
\partial_t u - \frac 1 2 |\nabla u|^2 + \frac 1 2 \Delta u - p[\rho] = 0, & \hbox{ on } (0, T)\times \mathbb T^3 \\
u(T, x) = h[\rho](x) & \hbox{ on } \mathbb T^3.
\end{array}
\end{equation}
Interestingly, one can check that, 
\begin{itemize}
\item 
if $(\rho, u)$ solves \eqref{eq:fpk2}-\eqref{eq:hjb1}, 
then $(\rho, \nabla u)$ solves  \eqref{eq:fpk2}- \eqref{eq:momentum1}, and hence
$(\rho, \nabla u)(T-t, x)$ solves \eqref{eq:nse4}. 
\end{itemize}
Therefore, to prove the solvability of \eqref{eq:nse4}, it is enough to prove the solvability of  \eqref{eq:fpk2}-\eqref{eq:hjb1}.

%%%%%%

\subsection{Probabilistic setting of MFG}
In this section, we provide a probabilistic setting for the MFG on a torus leading to \eqref{eq:fpk2}-\eqref{eq:hjb1}. 
We refer \cite{GPV16} for analytical MFG settings for the comparison.

\subsubsection{Generic player's controlled dynamic and Dynkin's formula on a torus}

Let $(\Omega, \mathbb P, \mathcal F, (\mathcal F_t)_{t\ge 0})$ be a filtered probability space satisfying the usual conditions with $W$ being a $\mathcal F_t$-adapted $ \mathbb R^3$-valued Brownian motion.
In this section, we set up a  mean field game on 
the state space $3$-torus $\mathbb T^3 = \mathbb R^3/\mathbb Z^3$.
A generic player's position $X^{v, t, \xi}$ at a velocity $-v$ perturbed by a white noise starting from initial time and position $(t, x)$  follows 
\begin{equation}
\label{eq:X}
X^{v, t, \xi}(r) = \xi - \int_t^r v(s, X^{v, t, \xi}_s) ds + W_r - W_t,
\end{equation}
where $X^{v, t, \xi}(t)$ and $\xi$ are random values in $\mathbb T^3$,    $v: [0, T]\times \mathbb T^3\mapsto \mathbb R^3$ is the player's control.
In \eqref{eq:X}, the equal sign is interpreted up to the coset, i.e. \eqref{eq:X} can be rewritten by
$$
X^{v, t, \xi}(r) = \xi - \int_t^r v(s, X^{v, t, \xi}_s) ds + W_r - W_t + \mathbb Z^3.
$$
%Define $\Phi_1[v]$ as the distribution flow of $X^v$. 

To facilitate the subsequent analysis, we also justify the Dynkin's formula on a diffusion \eqref{eq:X} defined on $\mathbb T^3$.
\begin{lemma}
\label{l:fpk1}
If $v:[0, T]\times \mathbb T^3 \mapsto \mathbb R^3$ is a continuous in both variables and Lipschitz 
on $\mathbb T^3$, then the SDE \eqref{eq:X} has a unique strong solution
provided by $X^{v, t, \xi} = \pi(X^{\bar v, t, \xi'})$, where $X^{\bar v, t, \xi'}$ is the $ \mathbb R^3$-valued random process given by
\begin{equation}
\label{eq:X3}
X^{\bar v, t, \xi'}(r)  = \xi' - \int_t^r \bar v(s, X^{\bar v,  t, \xi'}(s)) ds + W_r - W_t,
\end{equation}
whenever  $ \pi(\xi') = \xi$ and $\bar v(t, x) = v(t, \pi(x))$. Moreover, the following Dynkin's formula holds for any $f\in C^{1,2} ([0, T]\times \mathbb T^3, \mathbb R)$,
\begin{equation}
\label{eq:dynkin}
\mathbb E f(r, X^{v, t, \xi} (r)) = 
\mathbb E \left[ f(\xi) + \int_t^r (\partial_t f - v\cdot \nabla f + \frac 1 2 \Delta f)(s, X^{v, t, \xi}(s)) ds \right].
\end{equation}
\end{lemma}
\begin{proof}
There exists a unique solution for \eqref{eq:X3}
due to the Lipschitz continuity of $\bar v$ in $x$. Therefore, the existence of \eqref{eq:X} is provided by
$$X^{v, t, \xi} (r) = X^{\bar v, t, \xi'} (r) + \mathbb Z^3.$$
For the uniqueness, one shall check 
$\pi(X^{\bar v, t,  \xi'} (r)) = \pi(X^{\bar v, t, \xi''} (r))$ for all $r$ whenever $\pi(\xi') = \pi(\xi'') = \xi$.
Indeed, if $\xi'' = \xi' + z$ for some random variable $z\in \mathbb Z^3$, then 
one can directly verify $X^{\bar v, t, \xi''}(r) = X^{\bar v, t, \xi'}(r) + z$ using the periodicity of $\bar v(t, \cdot)$.

To show \eqref{eq:dynkin}, we use Ito formula on $\bar f (r, X^{\bar v, t, \xi'} (r))$ with the periodic function $\bar f(t, x) = f(t, \pi(x))$ and eliminate the martingale part since
$\bar v \cdot \nabla \bar f$ is uniformly bounded.
\end{proof}

The next moment estimations are taken from \cite{Kry80} on the solution of SDE, which will be useful in the subsequent proof.
\begin{lemma}
\label{l:moments}
Suppose $X_i$ for $i = 1, 2$ satisfies
$$d X_i(s) = - v(s, X_i(s)) ds + dW_s, \ X_i(0) = \xi_i.$$
Then, the following estimation holds:
$$\mathbb E \sup_{s\in [0, T]} |X_1(s) - X_2(s)|^2 \le \lambda(|v_1|_{0,1} \vee |v_2|_{0,1}) (|\xi_1 - \xi_2|^2 + |v_1 - v_2|_0^2)$$
for some $\lambda \in \Lambda$. In the above, $a\vee b$ is the maximum of $a$ and $b$.
\end{lemma}

\subsubsection{MFG formulation as a fixed point} \label{s:mfg}
In this part, we will define MFG Nash equilibrium 
via a composition of three operators $\Phi_1: \mathcal D_1 \mapsto \mathcal D_2$,
$\Phi_2: \mathcal D_2 \mapsto \mathcal D_3$, and
$\Phi_3: \mathcal D_3 \mapsto \mathcal D_1$, where
%$\mathcal D_1, \mathcal D_2, \mathcal D_3$ are some metric spaces given by
% to be specified later. 
%$\Phi := \Phi_3 \circ \Phi_2 \circ \Phi_1$. 
\begin{itemize}
\item $\mathcal D_1, \mathcal D_2, \mathcal D_3$ are metric spaces 
\begin{itemize}
\item
$\mathcal D_1$ is the set of  measure valued processes $\rho$ in  $C^{1/2} ([0, T], \mathcal P)$ 
with a topology induced by $d_{1T}(\cdot, \cdot)$;
\item $\mathcal D_2$ is the collection of all function pairs $(p, h)$ in 
$C^{1/2, 2}([0, T] \times \mathbb T^3) \times C^4(\mathbb T^3)$
with a topology induced by 
$|p|_{0, 2} + |h|_4;$
\item 
$\mathcal D_3$ is the collection of the function $v$ in $C^{1, 2} ([0, T]\times \mathbb T^3, \mathbb R^3)$
with a topology induced by $|\cdot|_0$.
\end{itemize}

\item
Let $\Phi_1: \mathcal D_1 \mapsto \mathcal D_2$ be defined by
$$\Phi_1[ \rho ] = (p, h) [\rho],$$
where $p$ and $h$ are functions given in the system  \eqref{eq:nse4}.% \eqref{eq:fpk2}- \eqref{eq:momentum1}.
\item
We define $\Phi_2: \mathcal D_2 \mapsto \mathcal D_3$ 
from the following control problem.

Let $J[p,h, v]$ be  the accumulated total cost of the player  given by 
$$J[p,h, v](t, x) = \mathbb E \left [\int_t^T
\Big(\frac 1 2 |v|^2 - p \Big)(s, X^{v, t, x}(s)) ds + h(X^{v, t, x} (T)) \Big| \mathcal F_t \right].$$
where $X^{v, t, x}$ is the the controlled process of \eqref{eq:X}, the function
$p: \mathbb R^+\times \mathbb T^3 \mapsto \mathbb R$ is a given running cost, and 
$h:  \mathbb T^3 \mapsto \mathbb R$ is a given terminal cost.
The objective of the player is to minimize the total cost. 

For any $(p, h) \in \mathcal D_2$, the $\Phi_2[p,h]$ is defined to be the optimal feedback control if it exists in the space $\mathcal D_3$, i.e.
the following optimality condition 
$$J [p,h, \Phi_2[p,h]](t, x) \le J[p,h, v](t, x), \ \forall (t, x) \in [0, T] \times \mathbb T^3$$
for all $v \in \mathcal D_3$.

\item Let $\mu\in \mathcal P$ be a given probability density on $\mathbb T^3$. 
We define  $\Phi_3^{\mu}: \mathcal D_3 \mapsto \mathcal D_1$ as the mapping from $v\in \mathcal D_3$ to  the 
corresponding distribution flow at a velocity $-v$ generated by  $\{X^{v, 0, \xi}_t: 0\le t \le T\}$ of 
\eqref{eq:X} for some initial state $\xi \in \mathcal F_0$
satisfying $\mathcal L(\xi) = \mu$, i.e.
$\Phi_3^{\mu} [v](t) = \mathcal L (X^{v, 0, \xi}_t)$ for $v\in \mathcal D_3$ and $t\in [0, T]$,  or equivalently
$$\langle \Phi_3^{\mu} [v](t), \phi\rangle = \mathbb E[ \phi(X_t^{v, 0, \xi}) ], 
\ \forall \phi \in C^\infty(\mathbb T^3, \mathbb R), t \in [0, T].$$

\end{itemize}

Note that, if the above three operators are well defined, then the composition  
$$ \Phi^{\mu} = \Phi_3^{\mu} \circ \Phi_2 \circ \Phi_1 $$ 
is a mapping from $\mathcal D_1$ to itself. 
Next, we will define the solution of MFG as the fixed point of $\Phi^{\mu}$.
\begin{definition}
\label{d:mfg}
Given an initial distribution $\mu \in \mathcal P_1$, 
the Nash equilibrium measure $\rho$ of the mean field game is the distribution flow satisfying the fixed point condition
$$\rho = \Phi^{\mu}  [\rho].$$
%Accordingly, $v = \Phi_2 \circ \Phi_1 [\rho]$ is said to be the optimal control.
\end{definition}

%%%%%%%%%%%%
\subsection{Estimates on the $\Phi_3^\mu$}
Recall the subsection \ref{s:mfg} that, 
the operator $\Phi_3^{\mu}[v]$ is the distribution flow of $X^{v, 0, \xi}$ given by the dynamic \eqref{eq:X} associated to a  control $v$ 
and a smooth initial distribution $\mathcal L(\xi) = \mu$.

\begin{lemma}
\label{l:fpk}
%Let $\delta\in (0, 1/2]$ be a constant.
The operator $\Phi_3^{\mu}$ is a well defined mapping from $\mathcal D_3$ to $\mathcal D_1$ satisfying
following estimations: There exists $\lambda \in \Lambda$, such that
$$|\Phi_3^{\mu} [v]|_{1T, 1/2} \le \lambda (|v|_0), \ \forall v\in \mathcal D_3, $$
and 
$$d_{1T} (\Phi_3^\mu(v_1), \Phi_3^\mu(v_2)) 
\le \lambda(|v_1|_{0, 1} \vee |v_2|_{0, 1}) |v_1 - v_2|_{0}.
$$
Furthermore, $\rho(t, x) = \Phi_3^{\mu} (t, x)$ is the unique classical solution of FPK  \eqref{eq:fpk2}.
\end{lemma}
\begin{proof}
If $v\in \mathcal D_3$, then there exists unique strong solution for SDE \eqref{eq:X}, and its distribution is 
the classical solution for FPK \eqref{eq:fpk2}. 
A similar calculation of \cite{Car13}  yields
$$d_1(\rho(t), \rho(s)) \le (1+ \sqrt T |v|_0 ) |t -s|^{1/2},
$$
which implies 
$$[\rho]_{1T, 1/2} \le \lambda ( |v|_0 )$$
for some $\lambda \in \Lambda$.  Moreover,  
$$
|\rho|_{1T} = \sup_t \int_{{\mathbb T^3}} |x| \rho(t, x) dx \le
 \int_{{\mathbb T^3}} |x| \mu(x)dx + |v|_0 T + \sqrt T  
 \le \lambda ( |v|_0 )
$$
for possibly different  $\lambda \in \Lambda$.
Therefore, we conclude that  $\rho$ belongs to $\mathcal D_1$ 
satisfying the first estimate.

Suppose $X_i = X_i^{v_i, 0, \xi}$ is a solution of \eqref{eq:X} corresponding to $v_i \in \mathcal D_3$ and $\xi_1 = \xi_2$, 
and the distribution flow is denoted by $\mathcal L (X_i(t)) = \rho_i(t)$.
By Lemma \ref{l:moments}, the following calculations lead to the second estimation:
$$\begin{array}
{ll}
d_{1T} (\rho_1, \rho_2)&   \displaystyle = \sup_t d_1(\rho_1(t), \rho_2(t)) 
\\ & \displaystyle \le \sup_t \mathbb E |X_1(t) - X_2(t)| 
\\ & \displaystyle \le \lambda(|v_1|_{0,1} \vee |v_2|_{0,1}) |v_1 - v_2|_0.
\end{array}
$$
for some $\lambda \in \Lambda$.
\end{proof}

%%%%%%%%%%%%%%%%%%%%%%%%%%%%
\subsection{Estimates on $\Phi_2$} \label{s:phi2}

Recall subsection \ref{s:mfg} that,
$\Phi_2[p, h]$ is defined as the optimal feedback control, if it exists. 
\subsubsection{ Verification Theorem}
In the next, 
Lemma \ref{l:verification} shows that $\Phi_2$ is indeed well defined. 
Moreover, it provides the connection 
of the control problem, the associated HJB, and the momentum equation, which can be considered as an adapted version of verification theorem on a torus.
\begin{lemma}[Verification theorem]
\label{l:verification}
Suppose $(p, h) \in \mathcal D_2$, then 
there exists a unique  solution $u \in C^{1,3}([0, T]\times \mathbb T^3, \mathbb R)$ for the HJB equation
\begin{equation}\label{eq:hjb2}
\left\{
\begin{array}
{ll}
\partial_t u - \frac 1 2 |\nabla u|^2 + \frac 1 2 \Delta u - p = 0, & \hbox{ on } (0, T)\times \mathbb T^3 \\
u(T, x) = h(x) & \hbox{ on } \mathbb T^3.
\end{array}
\right.
\end{equation}
and $v = \nabla u$ is the unique classical solution of the momentum equation
\begin{equation}\label{eq:momentum2}
\left\{
\begin{array}
{ll}
\partial_t v - v\cdot \nabla v + \frac 1 2 \Delta v = \nabla p & \hbox{ on } (0, T)\times \mathbb T^3 \\
v(T, x) =  \nabla h(x) & \hbox{ on } \mathbb T^3.
\end{array}
\right.
\end{equation}
Moreover, the operator 
$\Phi_2: \mathcal D_2\mapsto \mathcal D_3$ defined from the control problem is well defined and satisfies
$$\Phi_2[p,h] (t,x)= v(t, x).$$ 
\begin{proof}
By Hopf-Cole transformation  $w = e^{-u}$, 
$u$ solves \eqref{eq:hjb2} if and only if $w$ solves
\begin{equation}
\label{eq:pde_w}
\left\{
\begin{array}
{ll}
\partial_t w + \frac 1 2 \Delta w + w p= 0, &  (0, T) \times \mathbb T^3 \\
w(T, x) = e^{-h(x)},  & \mathbb T^3. \\
\end{array}
\right.
\end{equation}
By Lemma 3.6 of \cite{CHHJS20}, there exists the unique $C^{1,3}$ solution $w$ for \eqref{eq:pde_w}.
In addition, $w$ takes strictly positive values by maximum principle.
Therefore, $u = - \ln w$ of \eqref{eq:hjb2} also belongs to $C^{1,3}$ and $v = - \nabla u$ is in $C^{1,2}$ and it solves \eqref{eq:momentum2}.

Next, we prove $\Phi_2[p,h] = \nabla u$. 
Let $X_s = X^{v, t, x}(s)$ be the solution of \eqref{eq:X} starting from $X_t = x$ with control $v$.
Since $u$ is in $C^{1,3}$, one can use Dynkin's formula of Lemma \ref{l:fpk1} and obtain
$$
\begin{array}
{ll}
\mathbb E [u(T, X_T)] & = u(t, x) + \mathbb E \left[ \int_t^T (\partial_t u + v \cdot \nabla u + \frac 1 2 \Delta u)(s, X_s) ds \Big| X_t = x\right] \\
& \ge u(t, x) - \mathbb E \left[ ( \frac 1 2 |v|^2 - p)(s, X_s) ds \right].
\end{array}
$$
In the above, we used 
$$v\cdot \nabla u + \frac 1 2 |v|^2 \ge - \frac 1 2 |\nabla u|^2. $$ 
Rearranging the above inequality, it leads to 
$u(t, x) \le J[p, h, v]$ and the inequality becomes equality only if $v = \nabla u$.
\end{proof}
\end{lemma}
\subsubsection{Estimations of $\Phi_2$}

 \begin{lemma}
\label{l:hjb}
%Let $M>0$.
For any $(p_i, h_i) \in \mathcal D_2$, $v_i = \Phi_2 [p_i, h_i]$, and
$$M_m = \max\{|p_i|_{0, m}, |h_i|_{m+2}: i = 1, 2\},$$ 
the following estimations hold: 
$$|v_1 - v_2|_0 \le \lambda(M_2)  (|h_1 - h_2|_4 + |p_1 - p_2|_{0,1})$$
and
$$|v_1|_{0, 1} \le \lambda(M_2) $$
for some  $\lambda \in \Lambda$.
\end{lemma}

\begin{proof}
%In this proof, $\kappa(\cdot)$ will be used as a generic positive increasing function.
We denote by $w_i = w[p_i, h_i]$ for $i = 1, 2$ 
with the solution map  $w[p, h]$ of \eqref{eq:pde_w}. 
By Lemma \ref{l:verification}, we know $v_i = \nabla u_i \in C^{1,2}$ with relation $w_i = e^{-u_i}$. 
Since \eqref{eq:pde_w} has a unique classical solution and it admits 
Feynman-Kac formula
$$w_i (t, x) = 
\displaystyle
\mathbb E \left[ \exp \left\{\int_t^T p_i(s, x + W_{s-t}) ds - h_i(x + W_{T-t}) \right\} \right], $$
we obtain the estimate
\begin{equation}
\label{eq:est1}
|w_i^{\pm 1}|_0 \le e^{T (|p_i|_0+ |h_i|_0)} \le \lambda(M_0) + \kappa, \ i = 1, 2.
\end{equation}
By Lemma 3.7 and Theorem 3.8 of \cite{CHHJS20}, 
we also have
\begin{equation}
\label{eq:est2}
\begin{array}
{ll}
|w_1 - w_2|_{0,1} & 
\displaystyle
\le \lambda(M_1) \cdot (|p_1 - p_2|_{0,1} + |h_1 - h_2|_3)
\end{array}
\end{equation}
and
\begin{equation}
\label{eq:est21}
|w_i|_{0,2} \le \lambda(M_2) 
\end{equation}
for some $\lambda \in \Lambda$. 

Note that, 
$$|v_1|_0 = \Big | \frac{\nabla w_1}{w_1} \Big|_0 \le \lambda (|w_1|_{0, 1} \wedge |w_1^{-1}|_0)  \le \lambda(M_1) .$$

Also from $\partial_{ij} u_1 = \frac{- \partial_j w_1 \partial_i u_1 - \partial_{ij} w_1}{w_1}$, we have the second estimate 
$$|v_1|_{0, 1} \le |w_1^{-1}|_0 (|w_1|_0 |v_1|_0 + |w_1|_{0,2})\le \lambda(M_2).$$

On the other hand, one can compute that 
$$
\begin{array}
{ll}
|v_1 - v_2|_{0} & \displaystyle= \Big | \frac{\nabla w_1}{w_1} - \frac{\nabla w_2}{w_2} \Big |_{0} \\
& \displaystyle
\le |w_1^{-1} w_2^{-1}|_{0} \cdot (|w_2|_{0} |\nabla w_1 - \nabla w_2|_{0} + 
|\nabla w_2|_{0} |w_1 - w_2|_{0})
\\ & \displaystyle
\le |w_1^{-1} w_2^{-1}|_{0} \cdot |w_2|_{0,1} \cdot 
 |w_1 - w_2|_{0, 1}).
\end{array}
$$
Therefore, 
the above estimates \eqref{eq:est1}, \eqref{eq:est2}, and \eqref{eq:est21} yield the first estimate.
\end{proof}

\subsection{The proof of the main result}
This subsection provides the proof of Theorem \ref{t:main}.
Recall that $\mathcal D_i$ and operators $\Phi_i$ are given in the subsection \ref{s:mfg}.

With assumptions \eqref{eq:a3}-\eqref{eq:a2} given 
in Theorem \ref{t:main}, $\Phi_1: \mathcal D_1 \mapsto \mathcal D_2$ is continuous. Moreover,
$\Phi_2: \mathcal D_2 \mapsto \mathcal D_3$ and $\Phi_3^\mu: \mathcal D_3 \mapsto \mathcal D_1$ are also continuous in view of
Lemma \ref{l:hjb} and Lemma \ref{l:fpk}, respectively. As a result, $\Phi^\mu: \mathcal D_1 \mapsto \mathcal D_1$ is continuous as a composition of three continuous mappings.

The assumption \eqref{eq:a1} together with Lemma \ref{l:hjb} and Lemma \ref{l:fpk}, 
we have uniform boundedness of $\Phi$ in the sense
$$|\Phi^\mu[\rho] |_{1T, 1/2} \le \lambda(\kappa),$$
where $\lambda\in \Lambda$ and $\kappa>0$. Therefore, if we consider a subset of $\mathcal D_1$ given by
$$\Gamma = \{\rho\in \mathcal D_1: |\rho|_{1T, 1/2} \le \lambda(\kappa)\},$$
then 
 $ \Phi^\mu$ is a mapping from $\Gamma$ into itself.
Note that,  $\Gamma$ is closed and bounded in $|\cdot|_{1T, 1/2}$, thus compact in $\mathcal D_1$.
Furthermore, $\Gamma$ is convex. 
The existence of MFG equilibrium is implied by Schauder's fixed point theorem.
Finally, Lemma \ref{l:verification} and Lemma \ref{l:fpk} implies that $(\rho, v)$ is solves the  \eqref{eq:fpk2} - \eqref{eq:momentum1} in the classical sense. Therefore, there exists a classical 
solution %$(\rho, v)$ 
for
the system \eqref{eq:nse4}.

\section{Further remarks} \label{s:remark}
In this paper, we showed the existence of the solution $(\rho, v)$ to the NSE  \eqref{eq:fpk2} - \eqref{eq:momentum1}, and further, it gives the solvability for Nash equilibrium and optimal control for the corresponding MFG. 
\subsection{Uniqueness}
There is no result on the uniqueness unless additional monotonicity is imposed. However, without further effort, it can be shown that $(\rho, v)$ is the solution to the NSE  \eqref{eq:fpk2} - \eqref{eq:momentum1} if and only if it is a pair of equilibrium and optimal control in the sense of 
Definition \ref{d:mfg}.
This is the major difference between the verification of MFG and control: the verification theorem of control yields the uniqueness of HJB, but the verification theorem of MFG yields only one-to-one correspondence between equilibrium and HJB-FPK.

\subsection{Periodic MFG on $ \mathbb R^3$}
Another remark is on the connection between MFG on the torus $\mathbb T^3$ and on Euclidean space $ \mathbb R^3$. 
It is usual practice that one can obtain periodic solution of NSE  \eqref{eq:fpk2} - \eqref{eq:momentum1} on $ \mathbb R^3$ if the system parameters are periodically extended correspondingly:
$$\tilde \mu(x) = \mu(\pi(x)), \ \ \forall x\in \mathbb R^3$$
and
\begin{equation}
\label{eq:p3}
\tilde p[\rho](t, x)  = \int_{\mathbb T^3} \bar p( \pi(x), \pi(y)) \rho(t, y) dy + \hat p(\pi(x)),\ \forall x, y\in \mathbb R^3
\end{equation}
\begin{equation}
\label{eq:h3} 
\tilde h[\rho](x) = \int_{\mathbb T^3} \bar h(\pi(x), \pi(y)) \rho_T(y) dy + \hat h(\pi(x)), \ \forall x, y\in \mathbb R^3.
\end{equation}
However, such an extension does not make sense for MFG on $ \mathbb R^3$, since $\tilde \mu$ is not a probability density anymore. 

Indeed, the MFG on torus is equivalent to MFG on Euclidean space with periodic cost functions. More precisely, consider periodic MFG on $ \mathbb R^3$ with
\begin{itemize}
\item
periodic cost functions $\tilde p$ of \eqref{eq:p3} and $\tilde h$ of \eqref{eq:h3}, and some (not periodic) initial density $\tilde \mu \in \mathcal P_1( \mathbb R^3)$,
\end{itemize}
the corresponding formulation of MFG on $\mathbb T^3$ is with
\begin{itemize}
\item
cost functions $p$ of \eqref{eq:p2} and $h$ of \eqref{eq:h2}, and the pushforward initial density $\mu = \pi_* \tilde \mu \in \mathcal P_1(\mathbb T^3)$.
\end{itemize}
With the solution $(\rho, v)$ of \eqref{eq:fpk2} - \eqref{eq:momentum1}, one can obtain the optimal control of the original MFG on $ \mathbb R^3$ by a periodic extension
$$\tilde v(t, x) = v(t, \pi(x)), \forall x\in \mathbb R^3,$$
and the equilibrium by solving
\begin{equation*}
\label{eq:fpk4}
\begin{array}
{ll}
\partial_t \rho - \nabla \cdot (\rho \tilde v) = \frac 1 2 \Delta \rho, &\hbox{ on } (0, T)\times \mathbb R^3 \\
\rho(0, x) =  \tilde \mu(x), & \hbox{ on } \mathbb R^3.
\end{array}
\end{equation*}
In summary, if MFG is given with periodic cost functions, then the optimal control is periodic.

\subsection{Hamiltonian system}
This approach follows the stochastic version of the Pontryayin maximum principle, see \cite{MY99b} and \cite{YZ99} for more details.
Recall that, if a smooth $u$ solves HJB given by
$$\partial_t u + \frac 1 2 u = H(t, x, - \nabla u), \ u(T, x) = h(x),$$
then, its gradient flow $Y(t) = - \nabla u(t, X(t))$ together with the optimal trajectory $X(t)$ forms a stochastic Hamiltonian system:
$$
\begin{array}
{ll}
d X_t = D_y H(t, X_t, Y_t) dt + d W_t, \ X_0 = x \\
d Y_t = - D_x H(t, X_t, Y_t) dt + Z_t d W_t, \ Y_T = - \nabla h(X_T).
\end{array}$$
From HJB \eqref{eq:hjb1}, the
Hamiltonian for the generic player's problem can be written as the sum of the kinetic energy and the potential energy:
$$H[\rho](t, x, y) = \frac 1 2 |y|^2 + p[\rho](t, x)$$
and it leads to the following McKean-Vlasov FBSDE system, so called Hamiltonian system of unknown $(X, Y, Z, \rho)$:
\begin{equation}
\label{eq:ham1}
\left\{
\begin{array}
{ll}
d X_t = Y_t dt + dW_t, \ X(0) = \xi, \\
d Y_t = - \nabla p[\mathcal L(X_t)](t, X_t) dt + Z(t) dW(t), \ Y(T) = - \nabla h[\rho](X_T).
%\mathcal L(\xi) = \mu.% \mathcal L(\rho(t)) = \mathcal L \sim X(t)
\end{array}
\right.
\end{equation}

\begin{proposition}
\label{l:ham1}
Hamiltonian system \eqref{eq:ham1} admits unique solution $(X, Y, Z)$.
Moreover,  there exists decoupling term $v\in C^{1,2}([0, T]\times \mathbb T^d, \mathbb T^d)$ satisfying 
$$Y_t = - v(t, X_t).$$
\end{proposition}
The proof is the direct consequence of the existence of MFG.

\vspace{.2in}
\noindent {\bf Acknowledgments.} 
Luo’s research was supported by a grant from the Research Grants Council of the Hong Kong Special Administrative Region, China (Project No. 11305818).

\bibliographystyle{plain}
%\bibliographystyle{plainnat}
%\bibliographystyle{apalike}
%\bibliography{../../refs}

\def\cprime{$'$}

\end{document}